\documentclass[a4paper,reqno]{amsart}
\usepackage{amssymb, amsmath, amscd}
\usepackage[final]{graphicx}
\newtheorem{theorem}{Theorem}[section]

\newtheorem{lemma}[theorem]{Lemma}

\newtheorem{definition}[theorem]{Definition}

\def\MM{\mathcal{M}}

\def\pone{\hbox{\rm\textsf{1}\hspace*{-0.9ex}\rule{0.15ex}{1.3ex}\hspace*{0.9ex}}}
\begin{document}
\title[Geodesic completeness of affine surfaces]
{Geodesic completeness and the quasi-Einstein equation for locally homogeneous affine surfaces}
\author{P. B. Gilkey and X. Valle-Regueiro}
\thanks{Research partially supported by Project MTM2016-75897-P (AEI/FEDER, UE)}
\address{PG: Mathematics Department, University of Oregon, Eugene OR 97403-1222, USA}
\email{gilkey@uoregon.edu}
\address{XV: Faculty of Mathematics,
University of Santiago de Compostela,
15782 Santiago de Compostela, Spain}
 \email{javier.valle@usc.es}

\begin{abstract}Let $\mathcal{M}$ be a Type~$\mathcal{A}$ affine surface.
We show that $\mathcal{M}$ is linearly strongly projectively flat.
We use the quasi-Einstein equation together with the condition that $\mathcal{M}$
is strongly projectively flat to examine to examine the geodesic completeness
of $\mathcal{M}$.
\end{abstract}
\maketitle
\section{Affine geometry}A pair $\mathcal{M}=(M,\nabla)$ is said to be an {\it affine surface} if
 $\nabla$ is a torsion free connection on the tangent bundle
of a smooth surface $M$. A map from one affine surface to another is said to be an {\it affine map} if it intertwines the two connections.
An affine surface is said to be {\it locally homogeneous} if
given any two points of the surface, there is the germ of an affine diffeomorphism 
taking one point to the other. 
Let $(x^1,x^2)$ be local coordinates on an affine surface. Adopt the {\it Einstein convention}
and sum over repeated indices to expand
$\nabla_{\partial_{x^i}}\partial_{x^j}=\Gamma_{ij}{}^k\partial_{x^k}$
in terms of the {\it Christoffel symbols}; the condition that
$\nabla$ is torsion free is equivalent to the symmetry $\Gamma_{ij}{}^k=\Gamma_{ji}{}^k$.
We have the following classification result due to Opozda~\cite{Op04}.
\begin{theorem}\label{T1}
Let $\mathcal{M}=(M,\nabla)$ be a locally homogeneous affine surface. At least one of the following
three possibilities holds for the local geometry:
\smallbreak\noindent$\mathcal{A}$. There exist local coordinates $(x^1,x^2)$ so that
$\Gamma_{ij}{}^k=\Gamma_{ji}{}^k$ is constant.
\smallbreak\noindent$\mathcal{B}$. There exist local coordinates $(x^1,x^2)$ so that
$\Gamma_{ij}{}^k=(x^1)^{-1}C_{ij}{}^k$ where we have $C_{ij}{}^k=C_{ji}{}^k$ constant.
\smallbreak\noindent$\mathcal{C}$. $\nabla$ is the Levi-Civita of the round sphere.
\end{theorem} 

We say that $\mathcal{M}$ is a Type~$\mathcal{A}$ model if
$\mathcal{M}=(\mathbb{R}^2,\nabla)$ where $\nabla$ is Type~$\mathcal{A}$, i.e the Christoffel symbols
$\Gamma_{ij}{}^k\in\mathbb{R}$. Let $\mathbb{R}^2$ be the group of translations acting on itself; a connection
$\nabla$ on $\mathbb{R}^2$ is Type~$\mathcal{A}$ if $\nabla$ is left-invariant, i.e. the translations are affine maps.
Since $\nabla$ is torsion free, $\Gamma_{12}{}^1=\Gamma_{21}{}^1$
and $\Gamma_{12}{}^2=\Gamma_{21}{}^2$.
Thus there are 6 free parameters and we may identify
the set of {\it Type~$\mathcal{A}$ models} with $\mathbb{R}^6$ by setting
$\mathcal{M}(a,b,c,d,e,f)=(\mathbb{R}^2,\nabla)$ where the Christoffel symbols are given by
$$\Gamma_{11}{}^1=a,\Gamma_{11}{}^2=b,
\Gamma_{12}{}^1=\Gamma_{21}{}^1=c,\Gamma_{12}{}^2=\Gamma_{21}{}^2=d,\Gamma_{22}{}^1=e,\Gamma_{22}{}^2=f\,.
$$
The notion of a Type~$\mathcal{B}$ or Type~$\mathcal{C}$
model is defined similarly.
The general linear group $\operatorname{Gl}(2,\mathbb{R})$ acts on the set of Type~$\mathcal{A}$ models by change of variables;
we say that two Type~$\mathcal{A}$ models are {\it linearly equivalent} if they differ
by a linear action.
There are surfaces which are both Type~$\mathcal{A}$ and Type~$\mathcal{B}$
which are not flat. Any such geometry is, up to linear equivalence, one of the structures  $\mathcal{M}_1^1$,
$\mathcal{M}_2^1(c_1)$, $\mathcal{M}_3^1(c_1)$, or $\mathcal{M}_4^1(c)$ to be described
presently in Definition~\ref{D1}; we refer to \cite{BGG18} for further details. The Type~$\mathcal{C}$ geometry is neither
Type~$\mathcal{A}$ nor Type~$\mathcal{B}$.
The curvature operator $R$ and
 the Ricci tensor $\rho$ of an affine surface are given by
\begin{eqnarray*}
&&R(\xi_1,\xi_2)=\nabla_{\xi_1}\nabla_{\xi_2}-\nabla_{\xi_2}\nabla_{\xi_1}-\nabla_{[\xi_1,\xi_2]},\\
&&\rho(\xi_1,\xi_2)=\operatorname{Tr}\{\xi_3\rightarrow R(\xi_3,\xi_1)\xi_2\}\,.
\end{eqnarray*}
In general, the Ricci tensor of an affine surface need not be symmetric. However,
in the Type~$\mathcal{A}$ setting, the Ricci tensor is symmetric and is given by
\begin{equation}\label{E1.1}\begin{array}{l}
\rho_{11}=(\Gamma_{11}{}^1-\Gamma_{12}{}^2) \Gamma_{12}{}^2
+\Gamma_{11}{}^2 (\Gamma_{22}{}^2-\Gamma_{12}{}^1),\\
\rho_{12}=\rho_{21}=\Gamma_{12}{}^1 \Gamma_{12}{}^2-\Gamma_{11}{}^2 \Gamma_{22}{}^1,\\
\rho_{22}=-(\Gamma_{12}{}^1)^2+\Gamma_{22}{}^2 \Gamma_{12}{}^1
+(\Gamma_{11}{}^1-\Gamma_{12}{}^2) \Gamma_{22}{}^1\,.
\end{array}\end{equation}

We say that a curve $\sigma$ in an affine surface is a {\it geodesic} if $\nabla_{\dot\sigma}\dot\sigma=0$,
i.e. $\ddot\sigma^i+\Gamma_{jk}{}^i\dot\sigma^j\dot\sigma^k=0$ for all $i$. If $\nabla$ is the Levi-Civita
connection of a Riemannian metric, geodesics locally minimize length. There is no such interpretation in 
affine geometry. An affine surface is said to be {\it geodesically complete}
if every geodesic $\sigma$ is defined for all $t\in\mathbb{R}$; otherwise the surface is said to be
 {\it geodesically incomplete}. We shall concentrate on the Type~$\mathcal{A}$ geometries
 so that the geodesic equation
is a pair of quadratic ODEs with constant coefficients. However, even with this restriction, it is still difficult
to solve these equations directly. Instead, we shall first discuss the notion of strongly projectively flat geometries
and show in Lemma~\ref{L1} that any Type~$\mathcal{A}$ geometry is strongly projectively flat. We shall then
introduce the quasi-Einstein equation and present its basic properties in Theorem~\ref{T2}. This will enable us to give
a classification of the Type~$\mathcal{A}$ geometries in Theorem~\ref{T3} which we will use to determine
which Type~$\mathcal{A}$ geometries are geodesically complete in Theorem~\ref{T5}; this gives a different
treatment of a result originally established by D'Ascanio et al.~\cite{DGP17} using different methods.

\section{Strongly projectively flat geometries}
Two affine connections $\nabla$ and $\tilde\nabla$
are said to be {\it projectively equivalent}\index{projectively equivalent} if there exists a smooth $1$-form $\omega$ so
\smallbreak\centerline{
$\nabla_XY= \tilde\nabla_XY+\omega(X)Y+\omega(Y)X\text{ for all }X,Y$.}
We remark that $\nabla$ and $\tilde\nabla$ have the same unparametrized geodesics if and only if they
are projectively equivalent (see Kobayashi and Nomizu~\cite{KN63}); reparametrization can, of course, affect geodesic completeness.
If $\omega=dg$ for some smooth function $g$,
then $\nabla$ and $\tilde\nabla$ are said to be {\it strongly projectively equivalent}. If $\mathcal{M}=(M,\nabla)$, then
we set ${}^g\mathcal{M}:=(M,\tilde\nabla)$ in this setting.
If $\nabla$ is strongly projectively equivalent to a flat connection, then 
$\mathcal{M}$ is said to be {\it strongly projectively flat}.
\begin{lemma}\label{L1}
Let $\mathcal{M}=(\mathbb{R}^2,\nabla)$ be a Type~$\mathcal{A}$ model. 
There exists a linear function $g(x^1,x^2)=a_1x^1+a_2x^2$ which provides
a strong projective equivalence from $\mathcal{M}$ to a flat Type~$\mathcal{A}$ model.
\end{lemma}

We remark that results of Eisenhart~\cite{E64} showed that an affine surface is strongly projectively flat if and only
if both $\rho$ and $\nabla\rho$ are symmetric. Let $\mathcal{M}$ be a Type~$\mathcal{A}$ model.
Equation~(\ref{E1.1}) shows that $\rho$ is symmetric and one can make a similar direct computation
to show $\nabla\rho$ is symmetric. However, this does not yield that the 1-form in question has constant
coefficients so Lemma~\ref{L1} does not follow from general theory.

\begin{proof} Let $\mathcal{M}=(\mathbb{R}^2,\nabla)$ be a Type~$\mathcal{A}$ model.
 We work modulo linear equivalence. We use Equation~(\ref{E1.1}) to study the Ricci tensor $\rho$ of $\mathcal{M}$. 
Let $g(x^1,x^2)=w_1x^1+w_2x^2$ for $(w_1,w_2)\in\mathbb{R}^2$ and let ${}^g\mathcal{M}$ be the resulting strong projective deformation.
We then have
$$\begin{array}{lll}
{}^g\Gamma_{11}{}^1=\Gamma_{11}{}^1+2w_1,&
{}^g\Gamma_{11}{}^2=\Gamma_{11}{}^2,&
{}^g\Gamma_{12}{}^1=\Gamma_{12}{}^1+w_2,\\[0.05in]
{}^g\Gamma_{12}{}^2=\Gamma_{12}{}^2+w_1,&
{}^g\Gamma_{22}{}^1=\Gamma_{22}{}^1,&
{}^g\Gamma_{22}{}^2=\Gamma_{22}{}^2+2w_2\,.
\end{array}$$
Let ${}^g\rho$ be the Ricci tensor of ${}^g\mathcal{M}$. In dimension 2, the Ricci tensor carries the geometry; ${}^g\mathcal{M}$ is flat
if and only if ${}^g\rho=0$.

\smallbreak\noindent{\bf Case 1.} Suppose $\Gamma_{11}{}^2\ne0$. Rescale $x^2$ to ensure $\Gamma_{11}{}^2=1$. We have
$$
{}^g\rho_{11}=-\Gamma_{12}{}^1 -( \Gamma_{12}{}^2)^2 + \Gamma_{22}{}^2 + w_1^2 + \Gamma_{11}{}^1 (\Gamma_{12}{}^2 + w_1) + w_2\,.$$
We set
$
w_2:=\Gamma_{12}{}^1 - \Gamma_{11}{}^1 \Gamma_{12}{}^2 + (\Gamma_{12}{}^2)^2 - \Gamma_{22}{}^2 - \Gamma_{11}{}^1 w_1 - w_1^2
$
to ensure ${}^g\rho_{11}=0$. Then
${}^g\rho_{12}=-w_1^3+O(w_1^2)$ and ${}^g\rho_{22}=(\Gamma_{11}{}^1 - \Gamma_{12}{}^2 + w_1){}^g\rho_{12}$.
Since ${}^g\rho_{12}$ is cubic in $w_1$,  we can find $w_1$ so $\rho_{12}=0$. This forces ${}^g\rho_{22}=0$.
\smallbreak\noindent{\bf Case 2.} Suppose $\Gamma_{11}{}^2=0$. We set $w_1=\Gamma_{12}{}^2-\Gamma_{11}{}^1$ and $w_2=-\Gamma_{12}{}^1$ to see ${}^g\rho(\mathcal{M})=0$.
\end{proof}

\section{The Quasi-Einstein Equation}
Let $\mathcal{H}f:=
(\partial_{x^i}\partial_{x^j}f-\Gamma_{ij}{}^k\partial_{x^k}f)\,dx^i\otimes dx^j$
be the {\it Hessian}. Let $\rho_s$ be
the symmetric Ricci tensor and let
$\mathcal{Q}:=\ker\{\mathcal{H}+\rho_s\}$. We refer to Brozos-V\'azquez et al.~\cite{BGGX18}
for a discussion of the context in which this operator arises and for
applications to 4-dimensional geometry arising from the modified Riemannian extension.
We refer to \cite{GV18} for the proof of the following result.

\begin{theorem}\label{T2}
If $\mathcal{M}$ is an affine surface, then. If $dg$ provides a strong projective equivalence between
$\mathcal{M}$ and ${}^g\mathcal{M}$, then $\mathcal{Q}({}^g\mathcal{M})=e^g\mathcal{Q}(\mathcal{M})$. 
We have that $\dim\{\mathcal{Q}(\mathcal{M})\}\le3$; equality holds if and only if $\mathcal{M}$
is strongly projectively flat.
If $\nabla$ and $\tilde\nabla$ are two strongly projectively flat structures on
a surface $M$, then $\nabla=\tilde\nabla$ if and only if $\mathcal{Q}(M,\nabla)=\mathcal{Q}(\tilde M,\nabla)$.
Suppose ${}^g\mathcal{M}$ is flat. Then we have that $\mathcal{Q}(\mathcal{M})=e^g\operatorname{Span}\{\pone,\phi^1,\phi^2\}$
and 
$\Phi:=(\phi^1,\phi^2)$ provides local coordinates so that the unparameterized
geodesics of $\mathcal{M}$ take the form $\Phi^{-1}(at+a_0,bt+b_0)$.
\end{theorem}

Define distinguished Type~$\mathcal{A}$ geometries and function spaces as follows. To simplify the notation,
let $\mathcal{S}(f_1,f_2,f_3):=\operatorname{Span}_{\mathbb{R}}\{f_1,f_2,f_3\}$.

\begin{definition}\label{D1}Let $c_1\notin\{0,-1\}$ and $c_2\ne0$. 
\medbreak$\MM_0^0:=\mathcal{M}(0,0,0,0,0,0)$,\hglue 2.1cm $\mathcal{Q}_0^0=\mathcal{S}(\pone,x^1,x^2)$
\smallbreak$\MM_1^0:=\mathcal{M}(1,0,0,1,0,0)$,\hglue 2.1cm  $\mathcal{Q}_1^0=\mathcal{S}(\pone,e^{x^1},x^2e^{x^1})$,
\smallbreak$\MM_2^0:=\mathcal{M}(-1,0,0,0,0,1)$,\hglue 1.9cm $\mathcal{Q}_2^0=\mathcal{S}(\pone,e^{x^2},e^{-x^1})$,
\smallbreak$\MM_3^0:=\mathcal{M}(0,0,0,0,0,1)$,\hglue 2.2cm $\mathcal{Q}_3^0=\mathcal{S}(\pone,x^1,e^{x^2})$,
\smallbreak$\MM_4^0:=\mathcal{M}(0,0,0,0,1,0)$,\hglue 2.2cm$\mathcal{Q}_4^0=\mathcal{S}(\pone,x^2,(x^2)^2+2x^1)$,
\smallbreak$\MM_5^0:=\mathcal{M}(1,0,0,1,-1,0)$,\hglue 2cm
        $\mathcal{Q}_5^0=\mathcal{S}(\pone,e^{x^1}\cos(x^2),e^{x^1}\sin(x^2))$,
\smallbreak$\MM_1^1:=\mathcal{M}(-1,0,1,0,0,2)$,\hglue 2cm $\mathcal{Q}_1^1=e^{x^2}\mathcal{S}(\pone,x^2,e^{-x^1})$,
\smallbreak$\MM_2^1(c_1):=\mathcal{M}(-1,0,c_1,0,0,1+2c_1)$,\hglue .35cm
$\mathcal{Q}_2^1(c_1))=e^{c_1x^2}\mathcal{S}(\pone,e^{x^2},e^{-x^1})$,
\smallbreak$\MM_3^1({c_1}):=\mathcal{M}(0,0,{c_1},0,0,1+2{ c_1})$, \hglue .51cm
     $\mathcal{Q}_3^1(c_1))=e^{c_1x^2}\mathcal{S}(\pone,e^{x^2},x^1)$,
\smallbreak$\MM_4^1(c):=\mathcal{M}(0,0,1,0,c,2)$,\hglue 1.8cm
     $\mathcal{Q}_4^1(c))=e^{x^2}\mathcal{S}(\pone,x^2,c(x^2)^2+2x^1)$,
\smallbreak$\MM_5^1(c):=\mathcal{M}(1,0,0,0,1+c^2,2c)$,
$\mathcal{Q}_5^1(c))=\mathcal{S}(e^{cx^2}\cos(x^2),e^{cx^2}\sin(x^2),e^{x^1})$
\smallbreak$\MM_1^2(a_1,a_2):=\mathcal{M}\left(\frac
{a_1^2+a_2-1,a_1^2-a_1,a_1a_2,a_1a_2,a_2^2-a_2,a_1+a_2^2-1}{a_1+a_2-1}\right)$,
\par\qquad$\mathcal{Q}_1^2(a_1,a_2))=\mathcal{S}(e^{x^1},e^{x^2},e^{a_1x^1+a_2x^2})$ for
 $a_1a_2\ne0$ and $a_1+a_2\ne1$,
\smallbreak$\MM_2^2(b_1,b_2):=\mathcal{M}\left(1+b_1,0,b_2,1,\frac{1+b_2^2}{b_1-1},0\right)$,
for $b_1\ne1$ and $(b_1,b_2)\ne(0,0)$,
\par\qquad $\mathcal{Q}_2^2(b_1,b_2)=\mathcal{S}(e^{x^1}\cos(x^2),e^{x^1}\sin(x^2),e^{b_1x^1+b_2x^2})$

\smallbreak$\MM_3^2(c_2):=\mathcal{M}(2,0,0,1,c_2,1)$, \hfill
     $\mathcal{Q}_3^2(c_2)=e^{x^1}\mathcal{S}(\pone,x^1-c_2x^2,e^{x^2})$,
\smallbreak$\MM_4^2(\pm1):=\mathcal{M}(2,0,0,1,\pm1,0)$,
$\mathcal{Q}_4^2(\pm1))=\mathcal{S}(e^{x^1},x^2e^{x^1},(2x^1\pm (x^2)^2)e^{x^1})$.
\end{definition}

\begin{theorem}\label{T3}
If $\mathcal{M}$ is a Type~$\mathcal{A}$ model, then $\mathcal{M}$ is linearly equivalent to one of
the models $\mathcal{M}_i^\nu(\cdot)$ of Definition~\ref{D1}.
We have that $\mathcal{Q}(\mathcal{M}_i^\nu(\cdot))=\mathcal{Q}_i^\nu(\cdot)$ and that 
the Ricci tensor of $\mathcal{M}_i^\nu(\cdot)$ has rank $\nu$.
\end{theorem}

\begin{proof}
Let $\mathcal{M}$ be a Type~$\mathcal{A}$ model. By Lemma~\ref{L1}, $\mathcal{M}$ is
strongly projectively flat. Thus by Theorem~\ref{T2}, $\mathcal{M}$ is determined by $\mathcal{Q}(\mathcal{M})$. Since the Christoffel
symbols of $\mathcal{M}$ are constant, the translation group acts by affine diffeomorphisms. This implies that 
$\partial_{x^1}$ and $\partial_{x^2}$ are affine Killing vector fields. Consequently,
$\mathcal{Q}(\mathcal{M})$ is a finite dimensional $\partial_{x^1}$ and $\partial_{x^2}$ module. Let 
$\mathcal{Q}_{\mathbb{C}}(\mathcal{M}:=\mathcal{Q}(\mathcal{M})\otimes_{\mathbb{R}}\mathbb{C}$
be the complexification. This 3-dimensional space of functions invariant under the action of $\{\partial_{x^1},\partial_{x^2}\}$. By
examining the generalized simultaneious eigenvalues of  this action,  we can conclude that $\mathcal{Q}_{\mathbb{C}}(\mathcal{M})$
is generated by functions of the form $e^{a_1x^1+a_2x^2}p(x^1,x^2)$ where $p$ is polynomial and $(a_1,a_2)\in\mathbb{C}^2$.
With a bit of additional work, one can classify the possible solution spaces $\mathcal{Q}$ up to linear equivalence and
show they are linearly equivalent to $\mathcal{Q}_i^\nu(\cdot)$ for some value of the parameters;
we refer to~\cite{GV18} for further details.  By Theorem~\ref{T2},
$\dim\{\mathcal{Q}\{\mathcal{M}_i^\nu(\cdot)\}\le3$. A direct computation shows that 
$\mathcal{Q}_i^\nu(\cdot)\subset\mathcal{Q}(\mathcal{M}_i^\nu(\cdot))$ and thus equality holds for dimensional reasons.
Finally, a direct computation determines $\rho(\mathcal{M}_i^\nu(\cdot))$ and shows that the Ricci tensor has rank $\nu$.
\end{proof}

We have the following relations amongst the models $\mathcal{M}_i^\nu(\cdot)$.
\begin{theorem}\label{Thd}
The following are affine maps.
\begin{enumerate}
\item$\Phi_1^0(x^1,x^2):=(e^{x^1},x^2e^{x^1})$  embeds $\MM_1^0$ in $\MM_0^0$.
\item$\Phi_2^0(x^1,x^2):=(e^{x^2},e^{-x^1})$  embeds $\MM_2^0$ in $\MM_0^0$.
\item$\Phi_3^0(x^1,x^2):=(x^1,e^{x^2})$  embeds $\MM_3^0$ in $\MM_0^0$.
\item$\Phi_4^0(x^1,x^2):=(x^2,(x^2)^2+2x^1)$ defines $\MM_4^0\approx\MM_0^0$.
\item$\Phi_5^0(x^1,x^2)=(e^{x^1}\cos(x^2),e^{x^1}\sin(x^2))$  immerses $\MM_5^0$ in 
$\MM_0^0$.
\item$\Phi_1^1(x^1,x^2):=(e^{-x^1},x^2)$  embeds $\MM_1^1$ in $\MM_4^1(0)$.
\item$\Phi_2^1(x^1,x^2):=(e^{-x^1},x^2)$  embeds $\MM_2^1(c_1)$ in $\MM_3^1(c_1)$. 
\item$\Phi_3^1(x^1,x^2)\rightarrow(x^1e^{-x^2},-x^2)$ defines $\MM_3^1(c_1)\approx\MM_3^1(-c_1-1)$.
\item$\Phi_4^1(c)(x^1,x^2):=(x^1+\frac12c(x^2)^2,x^2)$  defines $\MM_4^1(c)\approx\MM_4^1(0)$.
\item$\Phi_5^1(x^1,x^2)=(x^1,-x^2)$ is an isomorphism $\MM_5^1(c)\approx\MM_5^1(-c)$.
\end{enumerate}\end{theorem}

\begin{proof}By Lemma~\ref{L1}, the Type~$\mathcal{A}$ models $\mathcal{M}_i^\nu(\cdot)$ are strongly projectively flat.
Thus, by Theorem~\ref{T2}, affine morphisms between them correspond to local diffeomorphisms which intertwine their
corresponding spaces $\mathcal{Q}$. One verifies immediately that this condition is satisfied by the maps $\Phi_i^j(\cdot)$ of
the Theorem and the desired result now holds.
\end{proof}

We can draw the following consequence.
\begin{lemma}
Let $\mathcal{M}$ be a Type~$\mathcal{A}$ flat geometry. Then $\mathcal{M}$ is geodesically complete
if and only if $\mathcal{M}$ is linearly equivalent to $\mathcal{M}_0^0$ or to $\mathcal{M}_4^0$.
\end{lemma}

\begin{proof}By Theorem~\ref{T3}, $\mathcal{M}$ is linearly equivalent to $\mathcal{M}_i^0$ for some $i$.
 $\mathcal{M}_1^0$, $\mathcal{M}_2^0$, and $\mathcal{M}_3^0$ have affine embeddings into $\mathcal{M}_0^0$
 which are not surjective;
 they are therefore not geodesically complete. $\mathcal{M}_4^0$ is affine diffeomorphic to the flat affine plane $\mathcal{M}_0^0$ and thus
 is geodesically complete. $\mathcal{M}_5^0$ has an affine immersion into $\mathcal{M}_0^0$ which is not surjective;
it is not geodesically complete.
\end{proof}

We use Theorem~\ref{T3} to express $\mathcal{Q}_i^\nu(\cdot)=e^{g}\operatorname{Span}\{\pone,\phi_1,\phi_2\}$
for $g$ linear. Let $\Phi=(\phi_1,\phi_2)$. By Theorem~\ref{T2},  the unparameterized geodesics of $\mathcal{M}_i^\nu(\cdot)$
take the form $\Phi^{-1}(a_0+a_1t,b_0+b_1t)$. This reduces the problem of finding the geodesics of $\mathcal{M}_i^\nu(\cdot)$
to solving a single ODE defining the reparametrization. This fact informed our subsequent investigations; we did not simply
proceed mechanically to solve the ODEs in question.
We say a Type~$\mathcal{A}$ model $\mathcal{M}$ {\it can be geodesically completed} if there is an affine
embedding of $\mathcal{M}$ in a homogeneous geodesically complete surface; otherwise $\mathcal{M}$ is said
to be {\it essentially geodesically complete}. The following is a useful criteria.
\begin{lemma}\label{L3}
Let $\mathcal{M}$ be a Type~$\mathcal{A}$ model. Assume there exists a geodesic $\sigma(t)$
for $t\in(t_-,t_+)$ so that $\lim_{t\rightarrow\tau}|\rho(\dot\sigma(t),\partial_{x^i})|=\infty$ where
$\tau=t_+<\infty$ or $\tau=t_->-\infty$.
Then $\mathcal{M}$ is essentially geodesically incomplete.
\end{lemma}

\begin{proof} Suppose to the contrary that there exists an affine surface $\MM_1$ which is locally modeled on $\mathcal{M}$.
Copy a small piece of the given geodesic $\sigma$ into $\MM_1$ to define a geodesic $\sigma_1$ in $\MM_1$. We
may assume without loss of generality that $\MM_1$ is simply connected and extend the vector field $\partial_{x^i}$ to a globally 
defined affine Killing vector field $X_i$ on $\MM_1$. Results of \cite{BGG18} show
that $\MM_1$ is real analytic. Thus the function $f(t):=\rho_{\MM}(\dot\sigma,\partial_{x^i})(t)$ defined for $t\in(t_-,t_+)$
extends to a real
analytic function $f_1(t):=\rho_{\MM_1}(\dot\sigma_1(t),X_i(t))$ for $t\in\mathbb{R}$. This is not possible since by assumption
$f(t)$ blows up at a finite value.
\end{proof}

If the Ricci tensor of a Type~$\mathcal{A}$ model $\mathcal{M}$ has rank 1, then
$\mathcal{M}$ is linearly equivalent to $\mathcal{M}_i^1(\cdot)$ for some value of the parameters. Thus it suffices to
study these examples.

\begin{lemma}\label{L4}
$\mathcal{M}_1^1$,  $\mathcal{M}_2^1(c_1)$ for $c_1\ne-\frac12$,
$\mathcal{M}_3^1(c_1)$ for $c_1\ne-\frac12$, $\mathcal{M}_4^1(c)$ for any $c$, and  
$\mathcal{M}_5^1(c)$ for $c\ne0$ are essentially geodesically incomplete.
$\mathcal{M}_3^1(-\frac12)$ is geodesically complete.
$\mathcal{M}_2^1(-\frac12)$ and $\mathcal{M}_5^1(0)$ can
be geodesically completed.\end{lemma}

\begin{proof}A direct computation shows
$$\begin{array}{ll}
\rho_{\mathcal{M}_1^4}=dx^2\otimes dx^2,&\rho_{\mathcal{M}_2^4(c_1)}=(c_1+c_1^2)dx^2\otimes dx^2,\\
\rho_{\mathcal{M}_3^4(c_1)}=(c_1+c_1^2)dx^2\otimes dx^2,\qquad&\rho_{\mathcal{M}_4^4(c)}=dx^2\otimes dx^2,\\
\rho_{\mathcal{M}_5^4(c)}=(1+c^2)dx^2\otimes dx^2.
\end{array}$$
We apply the criteria of Lemma~\ref{L3} with $\partial_{x^i}=\partial_{x^2}$ to study these geometries.
\smallbreak\noindent{\bf Case 1.} Let $\mathcal{M}=\mathcal{M}_1^1$. A direct computation shows
$\sigma(t)=(0,\frac12\log(t))$ is a geodesic for $t\in(0,\infty)$. Since $\lim_{t\rightarrow0}|\rho(\dot\sigma,\partial_{x^2})|=\infty$,
$\mathcal{M}_1^1$ is essentially geodesically incomplete.
\smallbreak\noindent{\bf Case 2.} Let $\mathcal{M}=\mathcal{M}_2^1(c_1)$ or
$\mathcal{M}=\mathcal{M}_3^1(c_1)$ for $c_1\ne-\frac12$. A direct computation shows
$\sigma(t):=(0,\frac1{1+2c}\log(t))$ is a geodesic for $t\in(0,\infty)$. Since we have that
$\lim_{t\rightarrow0}|\rho(\dot\sigma,\partial_{x^2})|=\infty$, 
$\mathcal{M}$ is essentially geodesically incomplete.
\smallbreak\noindent{\bf Case 3.} Let $\mathcal{M}=\mathcal{M}_3^1(-\frac12)$. Suppose $b\ne0$.
Let $\sigma_{a,b}(t)=(\frac ab(e^{bt}-1),bt)$. Then $\sigma$ is a geodesic with $\sigma(0)=(0,0)$ and $\dot\sigma(0)=(a,b)$.
If $b=0$, let $\sigma_{a,b}(t)=(at,0)$. Then $\sigma$ is a geodesic with $\sigma(0)=(0,0)$ and $\dot\sigma(0)=(a,0)$. Thus
every geodesic starting at $(0,0)$ extends for infinite time. Since $\mathcal{M}$ is homogeneous, $\mathcal{M}$ is geodesically
complete.
\smallbreak\noindent{\bf Case 4.} Let $\mathcal{M}=\mathcal{M}_2^1(-\frac12)$. A direct computation shows
$\sigma(t):=(-\log(t),0)$ is a geodesic for $t\in(0,\infty)$. This geodesic can not be continued to $t=0$ and thus
$\mathcal{M}_2^1(-\frac12)$ is geodesically incomplete.
By Theorem~\ref{Thd}, $\mathcal{M}_2^1(-\frac12)$ has an affine embedding in $\mathcal{M}_3^1(-\frac12)$. 
Thus by Case 3, $\mathcal{M}_2^1(-\frac12)$ can be
geodesically completed.
\smallbreak\noindent{\bf Case 5.} Let $\mathcal{M}=\mathcal{M}_4^4(c)$. Let
$\sigma(t):=(-\frac c8\log(t)^2,\frac12\log(t))$. A direct computation shows this is a geodesic for $t\in(0,\infty)$. Since 
$\lim_{t\rightarrow0}|\rho(\dot\sigma,\partial_{x^2})|=\infty$, 
$\mathcal{M}$ is essentially geodesically incomplete.
\smallbreak\noindent{\bf Case 6.} Let $\mathcal{M}=\mathcal{M}_5^1(c)$. Suppose that $c\ne0$.
A direct computation shows that
$\sigma(t)=(\log(\cos(\frac{\log (t)}{2 c}))+\frac{\log (t)}{2},\frac{\log (t)}{2 c})$
is a geodesic for $t\in(0,\infty)$. Since
$\lim_{t\rightarrow0}|\rho(\dot\sigma,\partial_{x^2})|=\infty$, 
$\mathcal{M}$ is essentially geodesically incomplete.  
\smallbreak\noindent{\bf Case 7.} If $c=0$, the curve $\sigma(t)=(\log(\cos(t)),t)$ is a geodesic for $\mathcal{M}_5^1(c)$
which does not extend to $\mathbb{R}$. Thus $\mathcal{M}_5^1(0)$ is geodesically incomplete.
We complete the proof by showing $\mathcal{M}_5^1(0)$ can be geodesically completed. Let 
$\mathcal{N}=(\mathbb{R}^2,\nabla)$ be the affine surface where the only non-zero Christoffel symbol of $\nabla$ is $\Gamma_{22}{}^1=x^1$.
We compute $\{\cos(x^2),\sin(x^2),x^1)\}\subset\mathcal{Q}(\mathcal{N})$ and thus by Theorem~\ref{T2} for dimensional
reasons we have $\mathcal{Q}$ is spanned by these elements and $\mathcal{N}$ is strongly projectively flat.
Let
$$
\Psi_{a,b,c,d}(x^1,x^2):=(e^{a}x^1+b\cos(x^2)+c\sin(x^2),x^2+d)\,.
$$
Then $T(a,b,c,d)^*\mathcal{Q}(\mathcal{N})=\mathcal{Q}(\mathcal{N})$ so
$T(a,b,c,d)$ is an affine diffeomorphism of $\mathcal{N}$. Since these diffeomorphisms act transitively on $\mathcal{N}$,
$\mathcal{N}$ is homogeneous. If $b\ne0$, let $\sigma_{a,b}(t):=(\frac ab\sin(bt),bt)$; this is a geodesic with 
$$\sigma_{a,b}(0)=(0,0)\text{ and }\dot\sigma_{a,b}(0)=(a,b)\,.
$$
If $b=0$, let $\sigma_{a,0}(t):=(at,0)$; this is a geodesic with $\sigma_{a,0}(0)=(0,0)$ and with $\dot\sigma_{a,0}(0)=(a,0)$. Thus $\mathcal{N}$
is geodesically complete at $(0,0)$ and, since $\mathcal{N}$ is homogeneous, $\mathcal{N}$ is geodesically complete.
The map $\Phi(x^1,x^2)=(e^{x^1},x^2)$ embeds $\mathbb{R}^2$ in $\mathbb{R}^2$ and satisfies
$\Phi^*\mathcal{Q}(\mathcal{N})=\mathcal{Q}(\mathcal{M}_5^1(0)$. Thus $\Phi$ is an affine embedding of $\mathcal{M}_5^1(0))$
in $\mathcal{N}$ so $\mathcal{N}$ provides the desired geodesic completion of $\mathcal{M}_5^1(0)$.
\end{proof}

We begin our discussion of the geometries where the Ricci tensor has rank 2 with the following result.
\begin{lemma}\label{L5} $\mathcal{M}_1^2(a_1,a_2)$, $\mathcal{M}_2^2(b_1,b_2)$ for $b_1\ne-1$, $\mathcal{M}_3^2(c_2)$,
and $\mathcal{M}_4^2(\pm1)$ are essentially geodesically incomplete.
\end{lemma}

\begin{proof}
A direct computation shows the Ricci tensor for the Type~$\mathcal{A}$ models $\mathcal{M}_i^2(\cdot)$ has rank 2.
Consequently the criteria of Lemma~\ref{L3} for essential geodesic incompleteness
is simply the existence of a geodesic so that $\lim_{t\rightarrow\tau}|\dot x(t)|=\infty$ or $\lim_{t\rightarrow\tau}|\dot y(t)|=\infty$ for
some finite value $\tau$.
\smallbreak\noindent{\bf Case 1.} Let $\mathcal{M}=\mathcal{M}_1^2(a_1,a_2)$. Let
$$\begin{array}{ll}
\sigma_1(t):=\log(t)\frac{(1-a_2,a_1)}{1+a_1-a_2}&\text{if }1+a_1-a_2\ne0,\\
\sigma_2(t):=\log(t)\frac{(a_2,1-a_1)}{1+a_2-a_1}&\text{if }1+a_2-a_1\ne0.
\end{array}$$
Since $(1+a_2-a_1)+(1-a_2+a_1)=2$, at least one of these curves is well defined. A direct computation
shows such a curve is a geodesic and hence $\mathcal{M}$ is essentially geodesically incomplete.
\smallbreak\noindent{\bf Case 2.} Let $\mathcal{M}=\mathcal{M}_2^2(b_1,b_2)$ for $b_1\ne-1$.
The curve $\sigma(t)=\frac1{1+b_1}(\log(t),0)$ is a geodesic. Consequently, $\mathcal{M}$ is
essentially geodesically complete.
\smallbreak\noindent{\bf Case 3.} Let $\mathcal{M}=\mathcal{M}_3^2(c_2)$ or $\mathcal{M}=\mathcal{M}_4^2(\pm)$. The curve
$\sigma(t)=\frac12(\log(t),0)$ is a geodesic; consequently, $\mathcal{M}$ is
essentially geodesically complete.
\end{proof}

Before considering the geometry $\mathcal{M}_2^2(-1,b_2)$, we must establish a preliminary result.
\goodbreak\begin{lemma}\label{L6} Let $P$ be a point of an affine manifold $\mathcal{M}$. 
Let $\sigma:[0,T)\rightarrow M$ be an affine geodesic. Suppose
$\lim_{t\rightarrow T}\sigma(t)=P$ exists. Then there exists $\epsilon>0$ so that $\sigma$ can be extended to the parameter range $[0,T+\epsilon)$
as an affine geodesic.
\end{lemma}

\begin{proof} 
Put a positive definite inner product $\langle\cdot,\cdot\rangle$ on $T_PM$ to act as a reference metric. 
Let $B_r$ be the ball of radius $r$ about the origin in $T_PM$. Since the exponential
map is a local diffeomorphism, we can use $\exp_P$ to identify $B_\varepsilon$ with a neighborhood of $P$ in $M$ for some small
$\varepsilon$. We use this identification to define a flat Riemannian metric near $P$ on $M$ so that $\exp_P$ is an isometry from $B_\varepsilon$ to 
$M$. Let $d(\cdot,\cdot)$ be the associated distance function on $M$. Let $B_r(P):=\exp_P(B_r)=\{Q:d(P,Q)\le r\}$ for $r\le\varepsilon$. 
Choose linear coordinates on $T_PM$
to put coordinates on $B_\varepsilon(P)$. This identifies $T_QM$ with $T_PM$ and extends $\langle\cdot,\cdot\rangle$ to $T_QM$
for $Q\in B_\varepsilon(P)$.
Compactness shows that there exists $0<\tau<\frac12\varepsilon$ so that if $Q\in B_{\frac\varepsilon2}(P)$ and
if $\xi\in T_QM$ satisfies $\|\xi\|=1$, then the geodesic $\sigma_{Q,\xi}(t):=\exp_Q(t\xi)$ exists for $t\in[0,\tau]$ and
belongs to $B_\varepsilon(P)$. By continuity, we can
choose $0<\delta<\frac14\tau$ so that if $Q\in B_\delta(P)$ and $\|\xi\|=1$, 
then $d(\sigma_{Q,\xi}(\tau),\sigma_{P,\xi}(\tau))<\frac\tau2$. Since
$d(P,\sigma_{P,\xi}(\tau))=\tau$, this implies $d(P,\sigma_{Q,\xi}(\tau))\ge\frac12\tau$. 
We conclude from these estimates that any non-trivial geodesic
which begins in $B_\delta(P)$ continues to exist at least until it exits from $B_{\frac12\tau}(P)$
and that it does in fact exit from $B_{\frac12\tau}(P)$.
We assumed $\lim_{t\rightarrow T}\sigma(t)=P$. Choose $T_0<T$ so $\sigma(T_0,T)\subset B_\delta(P)$. Then $\sigma$ continues to exist
until $\sigma$ exits from $B_{\frac12\tau}(P)$. Furthermore, $\sigma(T)=P$ and $\sigma$ extends to a geodesic 
defined on $(T_0,T+\epsilon)$ for some $\epsilon$.
\end{proof}

We complete our discussion with the following result.
\begin{lemma}\label{L7}
$\mathcal{M}_2^2(-1,b_2)$ is geodesically complete.
\end{lemma}

\begin{proof} Let $\mathcal{M}=\mathcal{M}_2^2(-1,b_2)$. 
Suppose, to the contrary, that $\MM$ is geodesically
incomplete. Let $\sigma$ be a geodesic in $\MM$ which is defined on a parameter range $(t_0,t_1)$ where $t_1<\infty$
(resp. $-\infty<t_0$)
which can not be extended to a parameter range $(t_0,t_1+\varepsilon)$ (resp. $t_0-\varepsilon$)
for any $\varepsilon>0$. By Lemma~\ref{L6},
this implies that $\lim_{t\downarrow t_0}\sigma(t)$ (resp. $\lim_{t\uparrow t_1}\sigma(t)$) does not exist.
We argue for a contradiction. The non-zero Christoffel
symbols of $\mathcal{M}$ are $\Gamma_{12}{}^1=b_2$, $\Gamma_{12}{}^2=1$, and $\Gamma_{22}{}^1=-\frac12(1+b_2^2)$. 
We work in the tangent bundle and introduce variables $u^1(t):=\dot x^1(t)$ and $u^2(t):=\dot x^2(t)$.
This yields the geodesic
equations
\begin{equation}\label{E1.2}
\dot u^1+2b_2u^1u^2-\textstyle\frac12(1+b_2^2)u^2u^2=0\text{ and }\dot u^2+2u^1u^2=0\,.
\end{equation}
If $u^2(s)=0$ for any $s\in(t_0,t_1)$, then $\dot u^1(s)=0$ and $\dot u^2(s)=0$. Consequently,
 $u^1(t)=u^1(s)$ and $u^2(t)=u^2(s)$ solves this ODE and $(u^1,u^2)$ is constant on
 the interval $(t_0,t_1)$. We may therefore assume $u^2$ does not change sign on the interval $(t_0,t_1)$. We
 want initial conditions $u^1(0)=a$ and $u^2(0)=b$. Let $\tau$ be an unknown function with $\tau(0)=1$. Set
\begin{eqnarray*}
&&\textstyle u^1(t):=e^{-b_2 \tau (t)} \left(\frac{1}{2} \left(-2 a b_2+b b_2^2+b\right) \sin (\tau (t))+a \cos (\tau (t))\right),\\
&&u^2(t):=e^{-b_2 \tau (t)} ((b b_2-2 a) \sin (\tau (t))+b \cos (\tau (t)))
\end{eqnarray*}
We then have $u^1(0)=a$ and $u^2(0)=b$. Equation~(\ref{E1.2}) then gives rise to a single
ODE to be satisfied:
$$
\dot\tau(t)=e^{-b_2 \tau (t)} (-2 a \sin (\tau (t))+b b_2 \sin (\tau (t))+b \cos (\tau (t)))
$$
or equivalently $\dot\tau(t)=u^2(\tau(t))$. Since $u^2$ does not change sign, $\tau(t)$ is restricted to a parameter interval of length at most $\pi$.
Thus $u^1$ and $u^2$ are bounded.
If $u^2$ is positive (resp negative), then $\dot\tau(t)$ is positive (resp. negative) and bounded so $\tau(t)$ is monotonically increasing
(resp. decreasing) and bounded
on the interval $(t_0,t_1)$. Thus $\lim_{t\downarrow t_0}\tau(t)$ and $\lim_{t\uparrow t_1}\tau(t)$ exist so
$\lim_{t\downarrow t_0}\dot\sigma(t)$ and $\lim_{t\uparrow t_1}\dot\sigma(t)$ exist. We integrate to conclude
$\lim_{t\downarrow t_0}\dot\sigma(t)$ and $\lim_{t\uparrow t_1}\dot\sigma(t)$ exist which provides the desired contradiction
and completes the proof. We remark that work of Bromberg and Medina~\cite{B05}  can also be used to establish this result.
\end{proof}

We summarize our results as follows; this was derived previously by D'Ascanio et al.~\cite{DGP17} using an entirely different approach.

\begin{theorem}\label{T5}
Let $\mathcal{M}$ be a Type~$\mathcal{A}$ affine surface.
\begin{enumerate}
\item Suppose $\mathcal{M}$ is flat. Then $\mathcal{M}$ is geodesically complete if and only if $\mathcal{M}$
is linearly equivalent to $\MM_0^0$ or to $\MM_4^0$. 
\item Suppose the Ricci tensor of $\mathcal{M}$ has rank 1. Then
$\mathcal{M}$ is geodesically complete if and only if $\mathcal{M}$ is linearly equivalent to $\MM_3^1(-\frac12)$.
If $\MM$ is linearly equivalent to $\MM_5^1(0)$, then $\MM$ is geodesically incomplete but has a geodesic completion $\mathcal{N}$. 
If $\MM$ is linearly equivalent to $\MM_2^1(-\frac12)$, then $\MM$ is geodesically incomplete but has the geodesic completion 
$\MM_3^1(-\frac12)$. Otherwise $\MM$ is essentially geodesically incomplete.
\item Suppose that the Ricci
tensor has rank $2$. If $\MM$ is linearly equivalent to $\MM_2^1(-1,b_2)$, then $\MM$ is geodesically complete.
Otherwise $\MM$ is essentially geodesically incomplete.
\end{enumerate}\end{theorem}

\noindent The geodesic structures of these models is pictured below
$$\begin{array}{cccc}
\MM_0^0&\MM_4^0&\MM_3^1(-\frac12)&\mathcal{N}\\
\includegraphics[height=2.5cm,keepaspectratio=true]{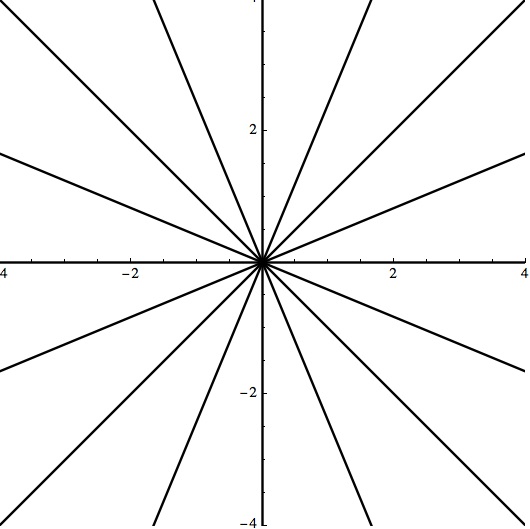}\quad&\quad
\includegraphics[height=2.5cm,keepaspectratio=true]{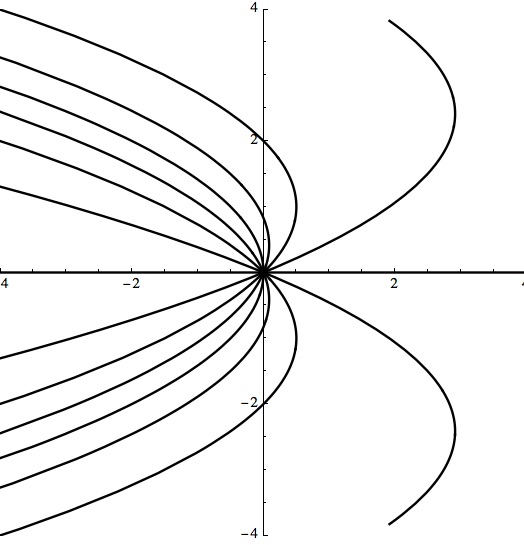}\quad&\quad
\includegraphics[height=2.5cm,keepaspectratio=true]{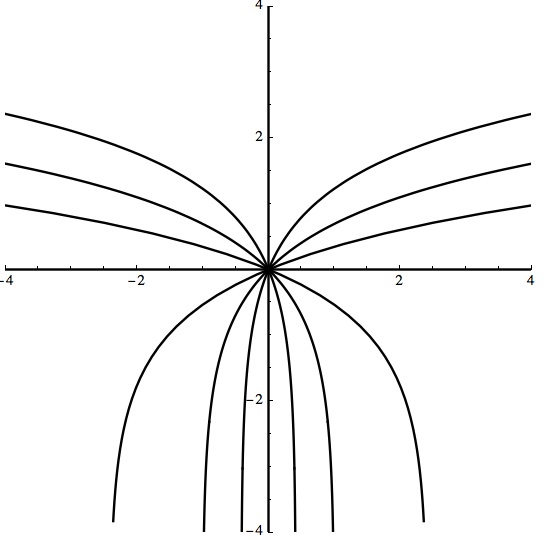}\quad&\quad
\includegraphics[height=2.5cm,keepaspectratio=true]{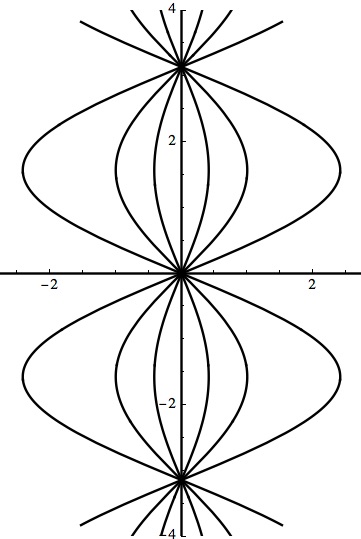}
\end{array}$$
$$
\begin{array}{ccc}
\MM_2^2(-1,0)&\MM_2^2(-1,1)&\MM_2^2(-1,2)\\
\includegraphics[height=2.5cm,keepaspectratio=true]{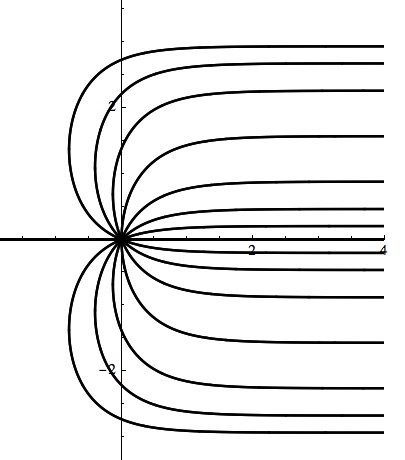}\ &\ 
\includegraphics[height=2.5cm,keepaspectratio=true]{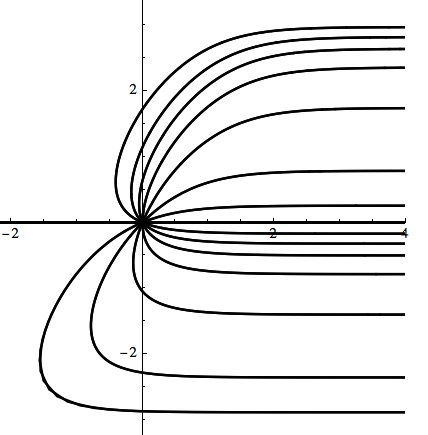}\ &\ 
\includegraphics[height=2.5cm,keepaspectratio=true]{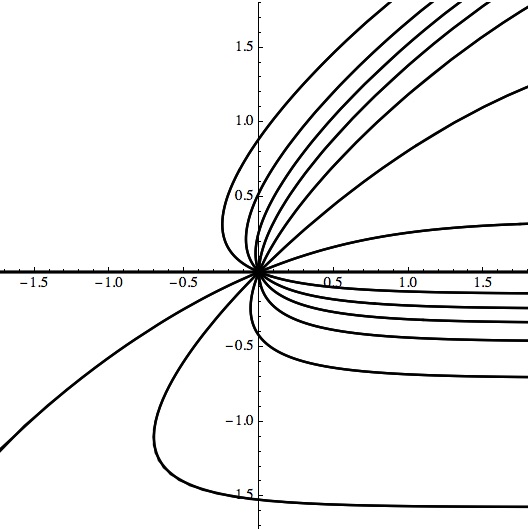}
\end{array}$$


\begin{thebibliography}{99.}
\bibitem{DGP17} D'Ascanio, D., Gilkey, P., and Pisani, P.:
Geodesic completeness for Type~A surfaces,
J. Diff. Geo. Appl. {\bf54} (2017), 31--43.
\bibitem{B05} Bromberg, S., and Medina, A.:
 A note on the completeness of homogeneous quadratic vector fields on the plane,
Qualitative Theory of Dynamical Systems {\bf6} (2005), 181--185.
\bibitem{BGG18}Brozos-V\'{a}zquez,  M, Garc\'{i}a-R\'{i}o, E.,  and Gilkey, P:
Homogeneous affine surfaces: affine Killing vector fields and gradient Ricci solitons, J. Math. Soc. Japan {\bf 70} (2018), 25--70.
\bibitem{BGGX18}Brozos-V\'{a}zquez,  M, Garc\'{i}a-R\'{i}o, E.,  Gilkey, P, and
and Valle-Regueiro, X.:
Half conformally flat generalized quasi-Einstein manifolds,
International Journal of Mathematics {\bf 29},  (2018) 1850002.
\bibitem{E64} Eisenhart, L. :Non-Riemannian geometry, 
AMS Colloquium Publications {\bf 8}, American Math. Society, Providence (1964).
\bibitem{KN63}
Kobayashi, S., and Nomizu, K.:
Foundations of Differential Geometry vol. I and \text{II},
{Wiley Classics Library}. A Wiley-Interscience Publication, John Wiley $\&$ Sons, Inc., New York, 1996.
\bibitem{GV18}
Gilkey, P. and Valle-Regueiro, X.:
Applications of PDEs to the study of affine surface geometry, Matematicki Vesnik (2018).
\bibitem{Op04}  Opozda, B.:
A classification of locally homogeneous connections on 2-dimensional manifolds,
Differential Geom. Appl. {\bf 21} (2004), 173--198.
\end{thebibliography}
\end{document}